\documentclass[12pt,leqno,a4paper]{amsart}
\usepackage{amssymb,enumerate}
\newcommand{\CC}{{\mathbb{C}}}
\newcommand{\FF}{{\mathbb{F}}}

\newcommand{\cE}{{\mathcal{E}}}
\newcommand{\cH}{{\mathcal{H}}}

\newcommand{\Irr}{{\operatorname{Irr}}}
\newcommand{\Hom}{{\operatorname{Hom}}}
\newcommand{\Soc}{{\operatorname{Soc}}}
\newcommand{\GL}{{\operatorname{GL}}}
\newcommand{\SL}{{\operatorname{SL}}}
\newcommand{\GU}{{\operatorname{GU}}}
\newcommand{\PCSp}{{\operatorname{PCSp}}}
\newcommand{\Sp}{{\operatorname{Sp}}}
\newcommand{\Spin}{{\operatorname{Spin}}}
\newcommand{\CO}{{\operatorname{CO}}}
\newcommand{\CSO}{{\operatorname{CSO}}}
\newcommand{\SO}{{\operatorname{SO}}}
\newcommand{\PCO}{{\operatorname{PCO}}}
\newcommand{\PCSO}{{\operatorname{PCSO}}}
\newcommand{\tr}{{\operatorname{tr}}}
\newcommand{\St}{{\operatorname{St}}}

\newcommand{\tw}[1]{{}^#1\!}
\newcommand{\hlf}{\frac{1}{2}}
\newcommand{\tG}{{\tilde G}}

\let\vhi=\varphi
\let\eps=\epsilon
\let\sqt=\boxtimes

\newtheorem{thm}{Theorem}[section]
\newtheorem{lem}[thm]{Lemma}
\newtheorem{cor}[thm]{Corollary}
\newtheorem{prop}[thm]{Proposition}

\newtheorem{thmA}{Theorem}

\theoremstyle{definition}

\theoremstyle{remark}
\newtheorem{rem}[thm]{Remark}

\begin{document}

\title[Low-dimensional representations of finite orthogonal
  groups]{Low-dimensional representations\\ of finite orthogonal groups}

\date{\today}

\author{Kay Magaard}
\address{}
\author{Gunter Malle}
\address{FB Mathematik, TU Kaiserslautern, Postfach 3049,
  67653 Kaisers\-lautern, Germany}
  \makeatletter\email{malle@mathematik.uni-kl.de}\makeatother

\thanks{The second author gratefully acknowledges financial support by SFB TRR
  195.}

\keywords{low dimensional representations, orthogonal groups, decomposition matrices}

\subjclass[1991]{Primary 20C33; Secondary 20D06, 20G40}

\begin{abstract}
We determine the smallest irreducible Brauer characters for finite quasi-simple
orthogonal type groups in non-defining characteristic. Under some restrictions
on the characteristic we also prove a gap result showing that the next larger
irreducible Brauer characters have a degree roughly the square of those of the
smallest non-trivial characters.
\end{abstract}

\maketitle


\section{Introduction}  \label{sec:intro}

This paper is devoted to studying low-dimensional irreducible representations
of finite orthogonal groups in non-defining characteristic. Our aim
is a gap result showing that there are a few well-understood representations
of very small degree, and all other irreducible representations have degree
which is roughly the square of the smallest ones. Knowing the low-dimensional
irreducible representations of quasi-simple groups has turned out to be of
considerable importance in many applications, most notably in the
determination of maximal subgroups of almost simple groups. More specifically
we prove:

\begin{thmA}   \label{thm:main even}
 Let $G=\Spin_{2n}^\eps(q)$ with $\eps\in\{\pm\}$, $q$ odd and $n\ge6$.
 Assume that $\ell\ge5$ is a prime not dividing $q(q+1)$. Let $\vhi$ be an
 $\ell$-modular irreducible Brauer character of $G$ of degree less than
 $q^{4n-10}-q^{n+4}$. Then $\vhi(1)$ is one of
 $$\begin{aligned}
   1,\quad &q\frac{(q^n-\eps1)(q^{n-2}+\eps1)}{q^2-1}-\kappa_1,\ 
   &q^2\frac{(q^{2n-2}-1)}{q^2-1}-\kappa_2,\\
   &\hlf\frac{(q^n-\eps1)(q^{n-1}\pm\eps1)}{q\mp 1},\ 
   &\frac{(q^n-\eps1)(q^{n-1}\pm\eps1)}{q\mp1},
 \end{aligned}$$
 where $\kappa_1,\kappa_2\in\{0,1\}$.
\end{thmA}

\begin{thmA}   \label{thm:main odd}
 Let $G=\Spin_{2n+1}(q)$ with $q$ odd and $n\ge5$. Assume that $\ell\ge5$ is
 a prime such that the order of $q$ modulo~$\ell$ is either odd, or bigger than
 $n/2$. Let $\vhi$ be an $\ell$-modular irreducible Brauer character of $G$
 of degree less than $(q^{4n-8}-q^{2n})/2$. Then $\vhi(1)$ is one of
 $$\begin{aligned}
  1,\quad & \frac{q^{2n}-1}{q^2-1},\quad \hlf q\frac{(q^n-1)(q^{n-1}-1)}{q+1},
    \quad \hlf q\frac{(q^n+1)(q^{n-1}+1)}{q+1},\\ 
   & \hlf q\frac{(q^n+1)(q^{n-1}-1)}{q-1}-\kappa_1,\quad
    \hlf q\frac{(q^n-1)(q^{n-1}+1)}{q-1}-\kappa_2,\quad 
   q\frac{q^{2n}-1}{q^2-1},\quad 
   \frac{q^{2n}-1}{q\pm1}
 \end{aligned}$$
 where $\kappa_1,\kappa_2\in\{0,1\}$.
\end{thmA}

Observe that the character degrees listed in the theorems are of the order of
magnitude about $q^{2n-2}$, $q^{2n-1}$ respectively, which is only slightly
larger than the square root of the given bound.

Gap results of the form described above have already been proved
for all other series of finite quasi-simple groups of Lie type. The situation
for orthogonal groups is considerably harder since the smallest dimensional
representations have comparatively much larger degree than for the other series.
For odd-dimensional orthogonal groups over fields of even characteristic
Guralnick--Tiep \cite{GT04} obtained gap results similar to ours without any
restriction on the non-defining characteristic $\ell$ for which the
representations are considered. Their approach crucially relies on the
exceptional isomorphism to symplectic groups.
\par
Our results do not cover all characteristics $\ell$ as our proofs rely on
unitriangularity of a suitable part of the $\ell$-modular decomposition matrix
of the groups considered which in turn is proved using properties of
generalised Gelfand--Graev characters. Since this has not been established in
full generality (although it is expected to hold), the present state of
knowledge makes it necessary to impose certain restrictions on the prime
numbers $\ell$ considered, as well as, more seriously, on the underlying
characteristic having to be odd.
\medskip

The paper is structured as follows. In Section~\ref{sec:complex} we determine
the small dimensional complex irreducible characters of spin groups using
Deligne--Lusztig theory. In Section~\ref{sec:parabolic} we investigate the
restriction of small dimensional Brauer characters to an end node parabolic
subgroup. Finally, with this information we determine the precise dimensions
of the smallest Brauer characters for all three series of spin groups in
Section~\ref{sec:main} and derive the gap results in Theorems~\ref{thm:main even} and~\ref{thm:main odd} including the precise values of the $\kappa_i$,
see Theorem~\ref{thm:BnSmall} and Corollary~\ref{cor:main Dn}.
\medskip

Kay and I started work on this paper around 2011. Sadly, he passed away very
unexpectedly shortly before the completion of the manuscript. I would like to
dedicate this paper to his memory.

\section{Small degree complex irreducible characters}   \label{sec:complex}

In this section we recall the classification of the smallest degrees of complex
irreducible characters of the finite spin groups $G$ by using Lusztig's
parametrisation in terms of Lusztig series $\cE(G,s)$ indexed by classes of
semisimple elements $s$ in the dual group $G^*$ ().

\subsection{The odd-dimensional spin groups $\Spin_{2n+1}(q)$}

Let $q$ be a power of a prime and $G=\Spin_{2n+1}(q)$ with $n\ge2$. Recall
that Lusztig's Jordan decomposition (see e.g. \cite[Thm.~2.6.22]{GM19}) gives a bijection
$$J_s:\cE(G,s)\longrightarrow \cE(C_{G^*}(s),1)$$
with unipotent characters of the centraliser $C_{G^*}(s)$, under which the
character degrees transform by the formula
$$\chi(1)=|G^*:C_{G^*}(s)|_{p'}\,J_s(\chi)(1).$$
\par
We start by enumerating unipotent characters of small degree. Here, we
allow for slightly larger degrees than in the general case, since this will
be needed later on and moreover we believe that this information may be of
independent interest.

Let us recall that a \emph{symbol} is a pair $S=(X,Y)$ of strictly increasing
sequences $X=(x_1<\ldots<x_r)$, $Y=(y_1<\ldots<y_s)$ of non-negative integers.
The \emph{rank} of $S$ is then defined to be
$$\sum_{i=1}^r x_i+\sum_{j=1}^s y_j
  -\left\lfloor\left(\frac{r+s-1}{2}\right)^2\right\rfloor.$$
The symbol $S'=(\{0\}\cup (X+1),\{0\}\cup (Y+1))$ is said to be
\emph{equivalent} to $S$, and so is the symbol $(Y,X)$. The rank is constant
on equivalence classes. The \emph{defect} of $S$ is
$d(S)=||X|-|Y||$, which clearly is also invariant under equivalence.
\par

The unipotent characters of the groups $\Spin_{2n+1}(q)$ are parametrised by
equivalence classes symbols of rank~$n$ and odd defect (see
e.g.~\cite[Thm.~4.5.1]{GM19}).
The following result is due to Nguyen \cite[Prop.~3.1]{Ng10} for $n\ge6$:

\begin{prop}   \label{prop:unipB}
 Let $G=\Spin_{2n+1}(q)$ or $\Sp_{2n}(q)$. Let $\chi$ be a unipotent character
 of $G$ of degree
 $$\chi(1)\le \begin{cases}
  q^{6n-16}-q^{4n-5}& \text{ when $n\ge6$},\\
  q^{15}-q^{12}& \text{ when $n=5$},\\
  q^{11}-q^9-q^8& \text{ when $n=4$},\\
  q^7-q^5& \text{ when $n=3$}.
 \end{cases}$$
 Then $\chi$ is as given in Table~\ref{tab:unipB} where we also record the
 degree of $\chi(1)$ as a polynomial in~$q$.
\end{prop}

\begin{table}[htbp]
 \caption{Small unipotent characters in types $B_n$ and $C_n$}
  \label{tab:unipB}
\[\begin{array}{|l|l|l|l|}
\hline
     S& \chi_S(1)& \deg_q(\chi_S(1))& \text{conditions}\cr
\hline
\binom{n}{-}& 1& 0& \cr
\hline
\binom{0,1,n}{-}& \hlf q\frac{(q^n-1)(q^{n-1}-1)}{q+1}& 2n-1& \cr
\binom{0,1}{n}& \hlf q\frac{(q^n+1)(q^{n-1}+1)}{q+1}& 2n-1& \cr
\binom{1,n}{0}& \hlf q\frac{(q^n+1)(q^{n-1}-1)}{q-1}& 2n-1& \cr
\binom{0,n}{1}& \hlf q\frac{(q^n-1)(q^{n-1}+1)}{q-1}& 2n-1& \cr
\hline
\binom{0,2,n-1}{-}& \hlf q^2\frac{(q^{2n}-1)(q^{n-1}-1)(q^{n-3}-1)}{q^4-1}& 4n-6& n>3\cr
\binom{0,2}{n-1}& \hlf q^2\frac{(q^{2n}-1)(q^{n-1}+1)(q^{n-3}+1)}{q^4-1}& 4n-6& n>3\cr
\binom{2,n-1}{0}& \hlf q^2\frac{(q^{2n}-1)(q^{n-1}+1)(q^{n-3}-1)}{(q^2-1)^2}& 4n-6& n>3\cr
\binom{0,n-1}{2}& \hlf q^2\frac{(q^{2n}-1)(q^{n-1}-1)(q^{n-3}+1)}{(q^2-1)^2}& 4n-6& n>3\cr
\binom{1,n-1}{1}& q^3\frac{(q^n+1)(q^n-1)(q^{2n-4}-1)}{(q^2-1)^2}& 4n-5& n>5\cr
\binom{0,1,2,n}{1}& \hlf q^4\frac{(q^n-1)(q^{2n-2}-1)(q^{n-2}-1)}{q^4-1}& 4n-4&
  n>5\text{ or }(n,q)=(3,2)\cr
\binom{0,1,2}{1,n}& \hlf q^4\frac{(q^n+1)(q^{2n-2}-1)(q^{n-2}+1)}{q^4-1}& 4n-4& n>5\cr
\binom{1,2,n}{0,1}& \hlf q^4\frac{(q^n+1)(q^{2n-2}-1)(q^{n-2}-1)}{(q^2-1)^2}& 4n-4& n>5\cr
\binom{0,1,n}{1,2}& \hlf q^4\frac{(q^n-1)(q^{2n-2}-1)(q^{n-2}+1)}{(q^2-1)^2}& 4n-4& n>5\cr
\hline
\binom{0,2}{2}& q^2\frac{q^6-1}{q^2-1}& 6& n=3\cr
\hline
\end{array}\]
\end{table}

\begin{proof}
The degree polynomials of unipotent characters for $n\le9$ can be computed
using {\tt Chevie} \cite{Chv}. For $q<20$, a direct evaluation of these
polynomials gives the claim. For $q>20$, the claim then follows by easy
estimates using the explicit formulas. For $n\ge10$ the assertion is shown
in \cite{Ng10}.
\end{proof}

We now enumerate the complex irreducible characters of $\Spin_{2n+1}(q)$ of
small degree. The irreducible character degrees of families of groups of fixed
Lie type over the field $\FF_q$ can be written as polynomials in $q$.
It turns out that the smallest such degree polynomials for groups of type
$B_n$ have degree in $q$ around $2n$, while the next larger ones have degree
in $q$ around $4n$. We list all irreducible characters whose degrees lie
in the first range. Note that the complex irreducible character of smallest
non-trivial degree for orthogonal groups was determined in \cite{TZ}.
For $n\ge5$, the following has been shown in \cite[Th.~1.2]{Ng10}.

\begin{thm}   \label{thm:lowBn}
 Let $G=\Spin_{2n+1}(q)$ with $n\ge3$. If $\chi\in\Irr(G)$ is such that
 $$\chi(1)< \begin{cases} q^{4n-8}& \text{ when $n\ge5$},\\
                          (q^{2n}-q^{2n-1})/2& \text{ when $n\in\{3,4\}$},
 \end{cases}$$
 then $\chi$ is as given in Table~\ref{tab:allBn}, where
 $1_G,\rho_1,\ldots,\rho_4$ are the first five unipotent characters listed
 in Table~\ref{tab:unipB}.
\end{thm}

\begin{table}[htbp]
 \caption{Smallest complex characters of $\Spin_{2n+1}(q)$, $n\ge3$}
  \label{tab:allBn}
\[\begin{array}{|l|l|l|l|l|}
\hline
 \chi& \chi(1)& \#\text{ ($q$ odd)}& \#\text{ ($q$ even)}& \deg_q(\chi(1))\cr
\hline
 1_G& 1& 1& 1& 0\cr
 \rho_{s,1}& (q^{2n}-1)/(q^2-1)& 1& 0& 2n-2\cr
 \rho_1& \hlf q(q^n-1)(q^{n-1}-1)/(q+1)& 1& 1& 2n-1\cr
 \rho_2& \hlf q(q^n+1)(q^{n-1}+1)/(q+1)& 1& 1& 2n-1\cr
 \rho_3& \hlf q(q^n+1)(q^{n-1}-1)/(q-1)& 1& 1& 2n-1\cr
 \rho_4& \hlf q(q^n-1)(q^{n-1}+1)/(q-1)& 1& 1& 2n-1\cr
 \rho_t^-& (q^{2n}-1)/(q+1)& (q-1)/2& q/2& 2n-1\cr
 \rho_{s,q}& q(q^{2n}-1)/(q^2-1)& 1& 0& 2n-1\cr
 \rho_t^+& (q^{2n}-1)/(q-1)& (q-3)/2& (q-2)/2& 2n-1\cr
\hline
\end{array}\]
\end{table}

\begin{proof}
For $3\le n\le 8$, the complete list of ordinary irreducible characters of $G$
and their degrees can be found on the website \cite{Lue}. For $q<30$, the claim
can then be checked by computer, while for $q>30$, an easy estimate, using the
known degrees in $q$ of the degree polynomials, shows that the given list is
complete. \par
For $n\ge9$ the result is in \cite{Ng10}. For later use let us recall the
origin of the various non-unipotent characters listed in Table~\ref{tab:allBn}
in Lusztig's parametrisation of characters in terms of semisimple classes in
the dual group $G^*=\PCSp_{2n}(q)$. 

Let $s\in G^*$ be an isolated involution with centraliser
$\Sp_{2}(q)\circ\Sp_{2n-2}(q)$. The corresponding Lusztig series $\cE(G,s)$ is
in bijective correspondence under Jordan decomposition with the unipotent
characters of $\Sp_{2}(q)\circ\Sp_{2n-2}(q)$. Thus, we obtain the semisimple
character $\rho_{s,1}$ and the character $\rho_{s,q}$ corresponding to the
Steinberg character in the $\Sp_{2}(q)$-factor (both given in
Table~\ref{tab:allBn}), while all other characters in that series have degree
at least $\rho_{s,1}(1)q(q^{n-1}-1)(q^{n-2}-1)/(q+1)/2$ (see
Table~\ref{tab:unipB}), which is larger than our bound. \par
The other characters arise from elements in $G^*$ with centraliser
$\Sp_{2n-2}(q)\times\GL_1(q)$ or $\Sp_{2n-2}(q)\times\GU_1(q)$. There are
$q-3$ central elements $t$ in $\GL_1(q)$ of order larger than~2, which are
fused to their inverses in $G^*$. The corresponding Lusztig series contain the
semisimple characters $\rho_t^+$ from Table~\ref{tab:allBn}. Moreover, the
$q-1$ elements in $\GU_1(q)$ of order larger than~2 give rise to the $(q-1)/2$
semisimple characters $\rho_t^-$. All other characters in these Lusztig
series have too large degree.
\end{proof}

\subsection{The even-dimensional spin groups $\Spin_{2n}^\pm(q)$}

For the even-dimensional spin groups $\Spin_{2n}^+(q)$ of plus-type, the
unipotent characters are parametrised by symbols of rank~$n$ and defect
$d\equiv0\pmod4$, while for those of minus-type, the parametrisation is by
symbols of defect $d\equiv2\pmod4$.

\begin{prop}   \label{prop:unipD}
 Let $\chi$ be a unipotent character of $\Spin_{2n}^+(q)$ of degree
 $$\chi(1)\le \begin{cases}
    (q^{6n-16}-q^{6n-17})/2& \text{ when $n\ge8$},\\
    q^{4n-5}-q^{4n-7}& \text{ when $4\le n\le 7$}.
 \end{cases}$$
 Then $\chi$ is as given in Table~\ref{tab:unipD}.
\end{prop}

\begin{table}[htbp]
 \caption{Small unipotent characters in type $D_n$}
  \label{tab:unipD}
\[\begin{array}{|l|l|l|l|}
\hline
     S& \chi_S(1)& \deg_q(\chi_S(1))& \text{conditions}\cr
\hline
\binom{n}{0}& 1& 0& \cr
\binom{n-1}{1}& q\frac{(q^n-1)(q^{n-2}+1)}{q^2-1}& 2n-3& \cr
\binom{1, n}{0, 1}& q^2\frac{q^{2n-2}-1}{q^2-1}& 2n-2& \cr
\hline
\binom{n-2}{2}& q^2\frac{(q^n-1)(q^{2n-2}-1)(q^{n-4}+1)}{(q^2-1)(q^4-1)}& 4n-10& n>4\cr
\binom{0, 1, 2, n-1}{-}& \hlf q^3\frac{(q^n-1)(q^{n-1}-1)(q^{n-2}-1)(q^{n-3}-1)}{(q+1)^2(q^2+1)}& 4n-7& \cr
\binom{0, n-1}{1, 2}& \hlf q^3\frac{(q^n-1)(q^{n-1}-1)(q^{n-2}+1)(q^{n-3}+1)}{(q^2-1)^2}& 4n-7& \cr
\binom{1, n-1}{0, 2}& \hlf q^3\frac{(q^n-1)(q^{n-1}+1)(q^{n-2}-1)(q^{n-3}+1)}{(q-1)^2(q^2+1)}& 4n-7& \cr
\binom{2, n-1}{0, 1}& \hlf q^3\frac{(q^n-1)(q^{n-1}+1)(q^{n-2}+1)(q^{n-3}-1)}{(q^2-1)^2}& 4n-7& \cr
\binom{1, 2, n}{0, 1, 2}& q^6\frac{(q^{2n-2}-1)(q^{2n-4}-1)}{(q^2-1)(q^4-1)}& 4n-6& \cr
\hline
\binom{n-3}{3}& q^3\frac{(q^n-1)(q^{2n-2}-1)(q^{2n-4}-1)(q^{n-6}+1)}{(q^2-1)(q^4-1)(q^6-1)}& 6n-21& n>6\cr
\hline
\binom{2}{2}& q^2\frac{q^6-1}{q^2-1}\qquad (2\times)& 6& n=4\cr
\binom{1,2}{1,2}& q^6\frac{q^6-1}{q^2-1}\qquad (2\times)& 10& n=4\cr
\binom{0, 3}{1, 3}& q^4\frac{(q^8-1)(q^5-1)}{(q-1)(q^2-1)}& 14& n=5,\ q>2\cr
\binom{0,1,2,3,4}{1}& \hlf q^7(q^5-1)(q^3-1)(q-1)^2& 17& n=5,\ q=2\cr
\binom{3}{3}& q^3\frac{(q^4+1)(q^{10}-1)}{q^2-1}\qquad (2\times)& 15& n=6\cr
\binom{0, 1, 3, 4}{-}& \hlf q^4\frac{(q^{10}-1)(q^3-1)^2(q-1)}{q+1}& 20& n=6,\ q=2,3\cr
\binom{0, 1, 3, 5}{-}& \hlf q^4\frac{(q^{12}-1)(q^7-1)(q^5-1)(q-1)}{q^3+1}& 26& n=7\cr
\binom{4}{4}& q^4\frac{(q^{14}-1)(q^{10}-1)(q^6+1)}{(q^2-1)(q^4-1)}\qquad (2\times)& 28& n=8\cr
\binom{5}{4}& q^4\frac{(q^9-1)(q^8+1)(q^{14}-1)(q^6+1)}{(q-1)(q^4-1)}& 36& n=9,\ q>2\cr
\hline
\end{array}\]
\end{table}

\begin{proof}
The (more involved) case $n\le8$ can be handled computationally as indicated
in the proof of Proposition~\ref{prop:unipB}. For $n\ge9$ this is shown
by a slight variation of the arguments used in the proof of
\cite[Prop.~3.4]{Ng10} which gives the list of unipotent characters of degree
at most $q^{4n-10}$.
\end{proof}

Similarly we obtain:

\begin{prop}   \label{prop:unip2D}
 Let $\chi$ be a unipotent character of $\Spin_{2n}^-(q)$ of degree
 $$\chi(1)\le \begin{cases}
    (q^{6n-16}-q^{6n-17})/2& \text{ when $n\ge8$},\\
    q^{4n-5}-q^{4n-7}& \text{ when $4\le n\le 7$.}
 \end{cases}$$
 Then $\chi$ is as given in Table~\ref{tab:unip2D}.
\end{prop}

Again, see \cite[Prop.~3.3]{Ng10} for the list of unipotent characters of
degree at most $q^{4n-10}$.

\begin{table}[htbp]
 \caption{Small unipotent characters in type $\tw2D_n$}
  \label{tab:unip2D}
\[\begin{array}{|l|l|l|l|}
\hline
     S& \chi_S(1)& \deg_q(\chi_S(1))& \text{conditions}\cr
\hline
\binom{0, n}{-}& 1& 0& \cr
\binom{1, n-1}{-}& q\frac{(q^n+1)(q^{n-2}-1)}{q^2-1}& 2n-3& \cr
\binom{0, 1, n}{ 1}& q^2\frac{(q^{2n-2}-1)}{q^2-1}& 2n-2& \cr
\hline
\binom{2, n-2}{-}& q^2\frac{(q^n+1)(q^{2n-2}-1)(q^{n-4}-1)}{(q^4-1)(q^2-1)}& 4n-10& n>4\cr
\binom{1, 2, n-1}{ 0}& \hlf q^3\frac{(q^n+1)(q^{n-1}+1)(q^{n-2}-1)(q^{n-3}-1)}{(q^2-1)^2}& 4n-7& \cr
\binom{0, 2, n-1}{ 1}& \hlf q^3\frac{(q^n+1)(q^{n-1}-1)(q^{n-2}+1)(q^{n-3}-1)}{(q-1)^2(q^2+1)}& 4n-7& \cr
\binom{0, 1, n-1}{ 2}& \hlf q^3\frac{(q^n+1)(q^{n-1}-1)(q^{n-2}-1)(q^{n-3}+1)}{(q^2-1)^2}& 4n-7& \cr
\binom{0, 1, 2}{ n-1}& \hlf q^3\frac{(q^n+1)(q^{n-1}+1)(q^{n-2}+1)(q^{n-3}+1)}{(q+1)^2(q^2+1)}& 4n-7& \cr
\binom{0, 1, 2, n}{ 1, 2}& q^6\frac{(q^{2n-2}-1)(q^{2n-4}-1)}{(q^2-1)(q^4-1)}& 4n-6& \cr
\hline
\binom{3, n-3}{-}& q^3\frac{(q^n+1)(q^{2n-2}-1)(q^{2n-4}-1)(q^{n-6}-1)}{(q^2-1)(q^4-1)(q^6-1)}& 6n-21& n>6\cr
\hline
\binom{4, 5}{-}& q^4\frac{(q^9+1)(q^8+1)(q^{14}-1)(q^6+1)(q-1)}{(q^2-1)(q^4-1)}& 36& n=9\cr
\binom{0, 1, 3}{ 3}& q^4\frac{(q^8-1)(q^5+1)}{(q+1)(q^2-1)}& 14& n=5\cr
\hline
\end{array}\]
\end{table}

The following two results have already been shown in
\cite[Th.~1.3 and~1.4]{Ng10} when $n\ge5$.

\begin{thm}   \label{thm:lowDn}
 Let $G=\Spin_{2n}^+(q)$. If $\chi\in\Irr(G)$ is such that
 $$\chi(1)< \begin{cases}
    q^{4n-10}& \text{ when $n\ge6$, or $n=5$ and $q$ is odd},\\
    q^{10}-q^8& \text{ when $n=5$ and $q$ is even},\\
    (q^8-2q^6)/4& \text{ when $n=4$ and $q$ is odd},\\
    q^8-q^7+q^5& \text{ when $n=4$ and $q$ is even},
 \end{cases}$$
 then $\chi$ is as given in Table~\ref{tab:allDn}, or $(n,q)=(4,2)$ and
 $\chi(1)=28$.
\end{thm}

\begin{proof}
For $3\le n\le 7$, the complete list of ordinary irreducible characters of $G$
and their degrees can be found on the website \cite{Lue}. For $q<50$, the claim
can then be checked by computer, while for $q>50$, an easy estimate shows that
the given list is complete. \par
For $n\ge8$ we refer to \cite[\S6,7]{Ng10}. Here, $1_g,\rho_1,\rho_2$ denote the 
first three unipotent characters listed in Table~\ref{tab:unipD}. The characters
$\rho_{s,a}^\pm,\rho_{s,b}^\pm$ are the semisimple characters in the Lusztig
series of involutions with disconnected centraliser of type
$\PCO_{2n-2}^\pm(q)$, and the characters $\rho_t^\pm$ are the semisimple
characters in the Lusztig series of semisimple elements with (connected)
centraliser of type $\PCSO_{2n-2}^\pm(q)$.
\end{proof}

\begin{table}[htbp]
 \caption{Smallest complex characters of $\Spin_{2n}^\eps(q)$, $n\ge4$}
  \label{tab:allDn}
\[\begin{array}{|l|l|l|l|l|}
\hline
 \chi& \chi(1)& \#\text{ ($q$ odd)}& \#\text{ ($q$ even)}& \deg_q(\chi(1))\cr
\hline
 1_G& 1& 1& 1& 0\cr
 \rho_1& q(q^n-\eps1)(q^{n-2}+\eps1)/(q^2-1)& 1& 1& 2n-3\cr
 \rho_{s,a}^-,\rho_{s,b}^-& \hlf(q^n-\eps1)(q^{n-1}-\eps1)/(q+1)& 2& 0& 2n-2\cr
 \rho_{s,a}^+,\rho_{s,b}^+& \hlf(q^n-\eps1)(q^{n-1}+\eps1)/(q-1)& 2& 0& 2n-2\cr
 \rho_t^-& (q^n-\eps1)(q^{n-1}-\eps1)/(q+1)& (q-1)/2& q/2& 2n-2\cr
 \rho_2& q^2(q^{2n-2}-1)/(q^2-1)& 1& 1& 2n-2\cr
 \rho_t^+& (q^n-\eps1)(q^{n-1}+\eps1)/(q-1)& (q-3)/2& (q-2)/2& 2n-2\cr
\hline
\end{array}\]
\end{table}

\begin{thm}   \label{thm:low2Dn}
 Let $G=\Spin_{2n}^-(q)$. If $\chi\in\Irr(G)$ is such that
 $$\chi(1)< \begin{cases}
    q^{4n-10}& \text{ when $n\ge6$},\\
    q^{10}-q^9& \text{ when $n=5$},\\
    (q^8-2q^6)/2& \text{ when $n=4$ and $q$ is odd},\\
    q^8-q^6& \text{ when $n=4$ and $q$ is even},
 \end{cases}$$
 then $\chi$ is as given in Table~\ref{tab:allDn}.
\end{thm}

\begin{proof}
Again, the case $n\le7$ can be settled using the data in \cite{Lue}, while for
$n\ge8$, we refer to \cite[\S6,7]{Ng10} (see Table~\ref{tab:unip2D} for the
unipotent characters). As before, $1_G,\rho_1,\rho_2$ denote the first three
unipotent characters listed in Table~\ref{tab:unip2D}. The characters $\rho_{s,a}^\pm,\rho_{s,b}^\pm$ lie in
the Lusztig series of involutions with disconnected centraliser of type
$\PCO_{2n-2}^\pm(q)$, and the $\rho_t^\pm$ are the semisimple characters in
the Lusztig series of semisimple elements with centraliser of type
$\PCSO_{2n-2}^\pm(q)$.
\end{proof}

\section{Locating Brauer characters of low degree}   \label{sec:parabolic}

Here we study the restriction of small dimensional irreducible $\ell$-Brauer
characters of spin groups to an end node parabolic subgroup. This requires no
assumptions on $\ell$ or on $q$. Throughout this section let
$G=\Spin_m^{(\pm)}(q)$ with $m\ge5$ and let $P= QL$ be a fixed maximal
parabolic subgroup of $G$ stabilising a singular $1$-space of the natural
module of $\SO_m^{(\pm)}(q)$, with unipotent radical $Q$ and Levi factor $L$.
Observe that $Q\cong \FF_q^{m-2}$ is the natural module for $L':=[L,L]$ of type
$\Spin_{m-2}^{(\pm)}(q)$. For $\chi \in Q^*:=\Irr(Q)=\Hom(Q,\CC^\times)$ we
denote by $L_\chi$ its inertia group in~$L$.

Let $k$ be an algebraically closed field of characteristic~$\ell$ not dividing
$q$. Let $W$ be a $kG$-module. Then the restriction of $W$ to $Q$ is
semisimple and we have a direct sum decomposition $W|_Q=\bigoplus_\chi W_\chi$
into the $Q$-weight spaces 
$$W_\chi:=\{w\in W \mid x.w=\chi(x) w \text{\ for all }x \in Q\}\qquad
  \text{for }\chi\in Q^*,$$
that is, the $Q$-isotypic components. The following notion was
introduced in \cite{MRT}: a $kG$-module $W$ is called \emph{$Q$-linear small}
if for all $\chi \in Q^*$ the simple $L_\chi'$-submodules of $\Soc(W_\chi)$ are
trivial. A module not satisfying this property is called \emph{$Q$-linear
large}.

The following is well-known:

\begin{lem}   \label{lem:Qchars}
 Let $\phi\in\Hom(\FF_q,\CC^\times)$ be a non-trivial linear character. Then
 \begin{enumerate}
  \item[\rm(a)]  $\sum_{a \in \FF_q} \phi(a) = 0$,
  \item[\rm(b)]  $\sum_{a \in \FF_q^\times} \phi(a) = -1$.
 \end{enumerate}
\end{lem}

\begin{proof}
Clearly (b) follows from (a). To see (a) note that the values of $\phi$ are
the $p$-th roots of unity, where $q$ is a power of $p$. Also $\phi$ is constant
on the cosets of $\ker(\phi)$, thus
$$\sum_{a\in\FF_q}\phi(a) = \sum_{k\in\FF_p}\sum_{a \in \ker(\phi)+k} \phi(a)
  = \sum_{k\in\FF_p} |\ker(\phi)|\phi(k) = |\ker(\phi)|\sum_{k\in\FF_p}\phi(k)
  = 0.$$
\end{proof}

\subsection{The odd-dimensional spin groups}

Let $G=\Spin_{2n+1}(q)$. We first recall the $L'\cong\Spin_{2n-1}(q)$-orbit
structure on $Q \cong \FF_q^{2n-1}$ and its dual. The $L'$-module $Q$ admits an
$L'$-invariant non-degenerate quadratic form $F$. Then two non-zero elements
$x_1,x_2$ of $Q$ lie in the same $L'$-orbit if and only if $F(x_1) = F(x_2)$.
Thus apart from the trivial orbit there is one orbit of singular vectors of
length $q^{2n-2}-1$, $(q-1)/2$ orbits of length $q^{2n-2}+q^{n-1}$ of
plus-type, and $(q-1)/2$ orbits of length $q^{2n-2}-q^{n-1}$ of minus-type.

If $W$ is a $kG$-module, then for any $\chi\in Q^*$ we thus obtain a direct
summand $W^{(\eps,\mu)}=\sum_{\psi\in\chi^{L'}}W_\psi$ of the socle of $[W,Q]$,
where $\eps\in\{0,\pm\}$ indicates the type of the stabiliser of $\chi$ and
$\mu$ is an $L_\chi'$-character (a constituent of $\Soc(W_\chi)|_{L_\chi'}$).
Denote  the Brauer character of $W^{(\eps,\mu)}$ by $\chi^{(\eps,\mu)}$. We
also write $W^\eps$ for the sum of all $W^{(\eps,\mu)}$. 

\begin{lem}   \label{lem:WChi B}
 Let $G = \Spin_{2n+1}(q)$ with $n\ge2$ and $q$ odd, and $x\in Q$ be a
 long root element of $G$. Then:
 \begin{enumerate}
  \item[\rm(a)] $\chi^{(0,\mu)}(x) = -1$, and
  \item[\rm(b)] $\chi^{(\pm,\mu)}(x) = \pm q^{n-1}$.
 \end{enumerate}
\end{lem}

\begin{proof}
We represent elements of $Q$ by row vectors and elements of its dual
$\Hom(Q,\FF_q)$ by column vectors. Note that as $Q$ is a self dual $L'$-module,
the $L'$-orbit structure on $Q$ and $\Hom(Q,\FF_q)$ is identical. We call the
elements of $\Hom(Q,\FF_q)$ functionals. So for example a singular functional
is an element of $\Hom(Q,\FF_q)$ on which the $L$-invariant quadratic form $F$
vanishes.

We choose a basis $\{e_1,\dots,e_{n-1},g,f_{n-1},\dots,f_1\}$ of $Q$ and
its dual basis in $\Hom(Q,\FF_q)$ in such a way that the Gram matrix of the
$L'$-invariant symmetric bilinear form with respect to this basis is the
matrix all of whose non-zero entries are $1$ and appear on the anti-diagonal.

Without loss we may assume that $x = [1,0,\dots,0]$ as all singular vectors
in $Q$ are $L'$-conjugate and $G$-conjugate to a long root element.
Let $\overline{t}=[a,\overline{b},c]^\tr\in\Hom(Q,\FF_q)$ with $a,c\in\FF_q$
and $\overline b\in\FF_q^{2n-3}$. Note that $\overline{t}(x) = a$.

Let $\phi\in\Hom(\FF_q,\CC^\times)$ be a non-trivial character. Then for each
$\chi \in Q^*$ there exists a unique $\overline{t}_\chi\in\Hom(Q,\FF_q)$ such
that $\chi(x)=\phi(\overline{t}_\chi(x))$.

So if $C\subseteq Q^*$, then the trace of $x\in Q$ on $\sum_{\chi\in C}W_\chi$
is equal to
$$\sum_{\chi\in C} \dim(W_\chi)\chi(x) =
  \sum_{\chi\in C} \dim(W_\chi) \phi(\overline{t}_\chi(x)).$$

We can now calculate the character values on $W^{(0,\mu)}$. First observe that
a functional $ \overline{t}  = [a,\overline{b},c]^\tr$ is singular if and only
if one of the following is true:
\begin{itemize}
\item[(A)] $F(\overline{b}) = ac =0$, or
\item[(B)] $F(\overline{b}) = -ac/2\neq 0$.
\end{itemize}
The number of $\overline{t}$ of type (B) is equal to the number of nonsingular
vectors in $\FF_q^{2n-3}$ which is $q^{2n-3}-q^{2n-4}$ times the number of
non-trivial choices for $a$ which is $q-1$. By Lemma~\ref{lem:Qchars}
these contribute $-(q^{2n-3} - q^{2n-4})$ to the trace of $x$ on $W^{(0,\mu)}$.

The elements $\overline{t}$ of type (A) come in two flavours depending on
whether or not $a = 0$. If $a = 0$, then if $c \neq 0 $ there are
$q^{2n-4}$ choices for $\overline{b}$, while if $c = 0 $ there are
$q^{2n-4} -1$ choices for $\overline{b}$. In total these $\overline{t}$
contribute
$$(q-1)q^{2n-4} + q^{2n-4} -1 = q^{2n-3} - 1$$
to the trace of $x$. Finally if $a \neq 0$, then $c=0$ while there are
$q^{2n-4}$ choices for $\overline{b}$ which yields a contribution of
$-q^{2n-4}$ to the trace. Thus the trace of $x$ on $W^{(0,\mu)}$ is
$$\chi^{(0,\mu)}= -(q^{2n-3} - q^{2n-4}) + ( q^{2n-3} - 1)  - q^{2n-4} = -1$$
as claimed.

Next we calculate the character value of $x$ on $W^{(+,\mu)}$. Observe that
the form $F$ evaluates to a fixed square, say $1$, on the functional
$\overline{t} = [a,\overline{b},c]^\tr$ if and only if
$F([a,\overline{0},c]^\tr) + F(\overline{b}) = 1$, that is, if and only if one
of the following is true:
\begin{itemize}
\item[(A)] $F(\overline{b}) = 1 $ and $ac =0$, or
\item[(B)] $1 - F(\overline{b}) = ac/2\neq 0$.
\end{itemize}
The contribution to the trace of $x$ by functionals of type~(A) occurs in one
of two ways: If $a \neq 0$, then $c= 0$ and then there are
$q^{2n-4} + q^{n-2}$ choices for $\overline{b}$ which yields
$$ -(q^{2n-4}+q^{n-2}).$$
If $a = 0$, then there are $q$ choices for $c$ and $q^{2n-4} + q^{n-2}$
choices for $\overline{b}$ which yields
$$q(q^{2n-4}+q^{n-2}).$$
To compute the contribution by functionals of type (B) we observe that
$a \neq 0$ and that for every choice of $a$ there are
$q^{2n-3} - (q^{2n-4} + q^{n-2})$ choices for $\overline{b}$ after which
$c$ is determined uniquely. Thus functionals of type (B) contribute
$$-q^{2n-3} + q^{2n-4} + q^{n-2} $$ to the trace of $x$. Summing up the
contributions yields that
$$\chi^{(+,\mu)}(x)
  = -(q^{2n-4}+q^{n-2}) + (q^{2n-3}+q^{n-1}) -q^{2n-3} + q^{2n-4} + q^{n-2}
  = q^{n-1}.$$

Finally we calculate the character value of $x$ on $W^{(-,\mu)}$.
Observe that $F$ evaluates to a fixed non-square $\alpha$ on
$ \overline{t}  = [a,\overline{b},c]^\tr$, if and only if
$F([a,\overline{0},c]^\tr) + F(\overline{b}) = \alpha$, that is, if and only
if one of the following is true:
\begin{itemize}
\item[(A)] $F(\overline{b}) = \alpha$ and $ac =0$, or
\item[(B)] $\alpha - F(\overline{b}) = ac/2\neq 0$.
\end{itemize}
The contribution by functionals of type (A) occurs in one of two ways: If
$a \neq 0$, then $c= 0$ and then there are $q^{2n-4} - q^{n-2}$ choices for
$\overline{b}$ which yields
$$ -(q^{2n-4}-q^{n-2}).$$
If $a = 0$, then there are $q$ choices for $c$ and $q^{2n-4} - q^{n-2}$
choices for $\overline{b}$ which yields
$$q(q^{2n-4}-q^{n-2}).$$
To compute the contribution by functionals of type (B) we observe that
$a \neq 0$ and that for every choice of $a$ there are
$q^{2n-3} - (q^{2n-4} - q^{n-2})$ choices for $\overline{b}$ after which
$c$ is determined uniquely. Thus functionals of type (B) contribute
$$-q^{2n-3} + q^{2n-4} - q^{n-2} $$
to the trace of $x$. Summing up the contributions yields that
$$\chi^{(-,\mu)}(x)
  = -(q^{2n-4}-q^{n-2}) + (q^{2n-3}-q^{n-1}) -q^{2n-3} + q^{2n-4} - q^{n-2}
  = -q^{n-1}$$
as claimed.
\end{proof}

\begin{rem}
While a similar result holds for the case of even $q$, we do not consider this
here as character bounds for $\Spin_{2n+1}(q)$ with $q$ even have already been
obtained in \cite{GT04}.
\end{rem}

We next compute the trace on a long root element in the Levi factor.

\begin{lem}   \label{lem:WchiL B}
 Let $G = \Spin_{2n+1}(q)$ with $n\ge2$ and $q$ odd. If $y\in L$ is a long
 root element then
 \begin{enumerate}
  \item[\rm(a)] $\chi^{(0,\mu)}(y) =  q^{2n-4}-1$, and
  \item[\rm(b)] $\chi^{(\pm,\mu)}(y) = q^{2n-4} \pm q^{n-1}$.
 \end{enumerate}
\end{lem}

\begin{proof}
Let $\chi\in Q^*$ be of type $\eps$. By definition
$W^{(\eps,\mu)} = \mu\uparrow_{P_\chi'}^{P'}$ where $\mu$ is a linear
character of $P_\chi'$. The element $y$ is unipotent and hence conjugate to an
element of $L_\chi'$ thus $\chi^{(\eps,\mu)}(y) = \chi^{(\eps,1)}(y)$
for all $\mu$. Hence it suffices to compute $\chi^{(\eps,1)}(y)$.

Now $\chi^{(\eps,1)}|_{L'}$ is the permutation character of $L'$ on the
cosets of $L_\chi'$. Thus $\chi^{(\eps,1)}(y)$ can be computed by counting
the fixed points of $y$ on the cosets of $L_\chi'$ in $L'$. This amounts
to counting vectors $v$ in $C_Q(y)$ with $F(v) = 0$, $F(v) = 1$, and $F(v)$ a
fixed non-square respectively. To make the count we observe that $C_Q(y)$ is
the orthogonal direct sum of a totally singular $2$-space with a non-degenerate
space of dimension $2n-5$.

Thus the number of singular non-zero vectors in $C_Q(y)$ is equal to
$q^2 q^{2n-6}-1 = q^{2n-4}-1$, the number of vectors in $C_Q(y)$ with
$F(y) =1$ is equal to $q^2 (q^{2n-6} + q^{n-3}) =  q^{2n-4} + q^{n-1}$,
while the number of vectors with $F(y)$ a fixed non-square
is equal to $q^2(q^{2n-6} - q^{n-3}) = q^{2n-4} - q^{n-1}$. The claim follows.
\end{proof}

\begin{prop}   \label{prop:QfixB}
 Let $G = \Spin_{2n+1}(q)$ with $n\ge2$ and $q$ odd. If $W$ is a $Q$-linear
 small $kG$-module then $C_W(Q) \neq \{0\}$.
\end{prop}

\begin{proof}
Let $x \in Q$ and $y\in L$ be long root elements of $G$, such that $x$ and $y$
are $G$-conjugate. As $x,y$ are $\ell'$-elements we can work with Brauer
characters.

Since $W$ is $Q$-linear small, we note that $W|_{P'}$ decomposes as
$C_W(Q)\oplus W^0\oplus W^+\oplus W^-$.
Denote the sum of the multiplicities of the characters $\chi^{(\eps,\mu)}$
in the character of $W^\eps$ by $a_\eps$, denote the character of
the $P'$-module $C_W(Q)$ by $\chi^c$ and the character of $W$ by $\phi$.
So with our notation
$\phi_{P'} = \chi^c + a_0\chi^0 + a_+\chi^+ + a_-\chi^-$.
Thus
$$ \phi(x) = \chi^c(1) -a_0 +a_+q^{n-1} - a_-q^{n-1} $$
by Lemma \ref{lem:WChi B} and
$$\phi(y) =
   \chi^c(y) + (a_0+a_+ +a_-)q^{2n-4}  - a_0 +a_+q^{n-1} - a_-q^{n-1}$$
by Lemma \ref{lem:WchiL B}.
As $x$ and $y$ are $G$-conjugate, $\phi(x) = \phi(y)$. Thus we find that
$$\chi^c(1) - \chi^c(y) = (a_0+a_+ +a_-)q^{2n-4} \geq q^{2n-4} > 0$$
as $a_0 + a_+ + a_->0$ (since $W$ is faithful) and so $C_W(Q)\ne\{0\}$.
\end{proof}

\begin{prop}   \label{prop:BnSmall}
 Let $G=\Spin_{2n+1}(q)$ with $n\ge2$ and $q$ odd and let $W$ be an irreducible
 $Q$-linear small $kG$-module. Then $W$ occurs as an $\ell$-modular composition
 factor of the Harish-Chandra induction from $L$ to $G$ of one of the modules
 in Table~\ref{tab:allBn}.
\end{prop}

\begin{proof}
As $W$ is $Q$-linear small Proposition~\ref{prop:QfixB} shows that
$C_W(Q)\neq 0$. An application of \cite[Lemma 4.2(ii) and (iii)]{GMST} then
gives that the $L$-constituents of $C_W(Q)$ are amongst those of $[W,Q]$.
Recall that the $P'$-module $[W,Q]$ is simply a sum of modules of the form
$W^{(\eps,\mu)}$. Thus the $L$-composition factors of the latter are precisely
those occurring in $W^{(\eps,\mu)}$. By assumption for all $\chi \in Q^*$ the
$L_\chi'$-submodules of $\Soc(W_\chi)$ are trivial, that is, any
$L'$-composition factors $\psi$ occurring in $W^{(\eps,\mu)}$ is a constituent
of an induced module $\mu \uparrow_{L'_\chi}^{L'}$, where $\mu$ is a linear
character of $L'_\chi$. In particular, $\psi(1)\le |L:L_\chi|\le q^{2n-2}-1$.
Then by Theorem~\ref{thm:lowBn}, $\psi$ is one of the modules in
Table~\ref{tab:allBn}.
\end{proof}

\subsection{The even-dimensional spin groups}
We now turn to the even dimensional spin groups $G=\Spin_{2n}^\eps(q)$,
$\eps\in\{\pm1\}$, with
$n\ge3$. Here, the $L'\cong \Spin_{2n-2}^\eps(q)$-orbit structure on
$Q \cong \FF_q^{2n-2}$ and its dual is as follows. Apart from the trivial
orbit there is one orbit of singular vectors of length
$q^{2n-3} +\eps( q^{n-1} - q^{n-2}) -1$ and $q-1$ orbits of length
$q^{2n-3} -\eps q^{n-2}$. When $q$ is odd, then half of the
$q-1$ orbits of length  $q^{2n-3} -\eps q^{n-2}$ are of plus type whereas
the others are of minus type (i.e., lie in distinct $L$-orbits).
When $q$ is even all orbits of length $q^{2n-3} -\eps q^{n-2}$ are in the
same $L$-orbit. As $Q$ is a self dual $L$-module, the $L$-orbit structures on
$Q$ and $\Hom(Q,\FF_q)$ are identical.

We first prove the analogue of Proposition~\ref{prop:QfixB}.

\begin{lem}   \label{lem:Wchi D}
 Let $G = \Spin_{2n}^\eps(q)$ with $n\ge3$, and $x\in Q$ be a long root element
 of $G$. Then
 \begin{enumerate}
  \item[\rm(a)] $\chi^{(0,\mu)}(x) = \eps (q^{n-1} - q^{n-2}) -1$, and
  \item[\rm(b)] $\chi^{(\neq 0,\mu)}(x) = -\eps q^{n-2}$.
 \end{enumerate}
\end{lem}

\begin{proof}
We argue as in Lemma~\ref{lem:WChi B} and keep the same notation. Choose
$\{e_1,\dots,e_{n-1},f_{n-1},\dots,f_1\}$, as basis of $Q$ and its dual
basis in such a way that the Gram matrix of the $L'$-invariant symmetric
bilinear form with respect to this basis is the matrix all of whose non-zero
entries are $1$ and appear on the anti-diagonal if $\eps=+$, while for
$\eps=-$, the Gram matrix is of this form except that the middle $2\times2$
square is not anti-diagonal.

Recall that we represent elements of $Q$ by row vectors and elements of
$\Hom(Q,\FF_q)$ by column vectors. Without loss we may assume that
$x = [1,0,\dots,0]$ as all singular vectors in $Q$ are $L'$-conjugate.
Let $\overline{t}=[a,\overline{b},c]^\tr$, where $a,c\in\FF_q$, and
$\overline{b}$ is an element of $\FF_q^{2n-4}$.

We start with the character values on $W^{(0,\mu)}$. Observe that a
functional $ \overline{t}  = [a,\overline{b},c]^\tr$ is singular if and only
if one of the following is true:
\begin{itemize}
\item[(A)] $F(\overline{b}) = ac =0$, or
\item[(B)] $F(\overline{b}) = -ac\neq 0$.
\end{itemize}
The number of $\overline{t}$ of type (B) is equal to the number of nonsingular
vectors in $\FF_q^{2n-4}$ which is
$q^{2n-4}-q^{2n-5} - \eps(q^{n-2} - q^{n-3})$
times the number of non-trivial choices for $a$ which is $q-1$.
By Lemma~\ref{lem:Qchars} these contribute
$$-(q^{2n-4}-q^{2n-5} - \eps( q^{n-2}  - q^{n-3}))$$
to the trace of $x$ on $W^{(0,\mu)}$.

The elements $\overline{t}$ of type (A) come in two flavours depending on
whether or not $a = 0$. If $a = 0$, then if $c \neq 0 $ there are
$q^{2n-5} +\eps(q^{n-2} - q^{n-3})$ choices for $\overline{b}$, while if
$c = 0 $ there are $q^{2n-5}  +\eps( q^{n-2} - q^{n-3}) -1$ choices for
$\overline{b}$. In total these $\overline{t}$ contribute
$$(q-1)( q^{2n-5} + \eps(q^{n-2} - q^{n-3})) +  (q^{2n-5} + \eps(q^{n-2}
  - q^{n-3}) -1 = q^{2n-4} + \eps(q^{n-1} - q^{n-2})-1$$
to the trace of $x$. Finally if $a \neq 0$, then $c=0$ while there are
$q^{2n-5} + \eps(q^{n-2} - q^{n-3})$ choices for $\overline{b}$ which yields a
contribution of $-( q^{2n-5} + \eps(q^{n-2} - q^{n-3}))$. Thus the trace of $x$
on $W^{(0,\mu)}$ is
$$\begin{aligned}
  -q^{2n-4}-q^{2n-5} + \eps(q^{n-2} - q^{n-3}) +&
  q^{2n-4} + \eps(q^{n-1}-q^{n-2}-1) - (q^{2n-5} + \eps(q^{n-2}-q^{n-3}))\\
   & = \eps(q^{n-1} - q^{n-2}) -1
\end{aligned}$$
as claimed.

Next we calculate the character value of $x$ on $W^{(\neq 0,\mu)}$.
Note that all $L'$-orbits of nonsingular vectors in $Q$ are of length
$q^{2n-3} -\eps q^{n-2}$. We observe that the form $F$ evaluates to
$\alpha \neq 0$ on the functional $ \overline{t}  = [a,\overline{b},c]^\tr$
if and only if $F([a,\overline{0},c]^\tr) + F(\overline{b}) = \alpha$, that is,
if and only if one of the following holds:
\begin{itemize}
\item[(A)] $F(\overline{b}) = \alpha $ and $ac =0$, or
\item[(B)] $\alpha - F(\overline{b}) = ac \neq 0$.
\end{itemize}
The contribution to the trace of $x$ by functionals of type (A) occurs in one
of two ways: If $a \neq 0$, then $c= 0$ and then there are
$q^{2n-5} -\eps q^{n-3}$ choices for $\overline{b}$ which contributes
$$ -(q^{2n-5} -\eps q^{n-3}).$$
If $a = 0$, then there are $q$ choices for $c$ and $ q^{2n-5} -\eps q^{n-3}$
choices for $\overline{b}$ which contributes
$$q( q^{2n-5} -\eps q^{n-3}).$$

To compute the contribution by functionals of type (B) we observe that
$a \neq 0$ and that for every choice of $a$ there are
$q^{2n-4} - q^{2n-5} + q^{n-3}$ choices for $\overline{b}$ after which $c$ is
determined uniquely. Thus functionals of type (B) contribute
$$-q^{2n-4} + q^{2n-5} - q^{n-3} $$
to the trace of $x$. Summing up the contributions yields that
$$\chi^{(\ne0,\mu)}(x)
  = -(q^{2n-5} -\eps q^{n-3})+ (q^{2n-4} -\eps q^{n-2})-q^{2n-4} + q^{2n-5}
    - q^{n-3} = -\eps q^{n-2}$$
as claimed.
\end{proof}

\begin{lem}   \label{lem:WchiL D}
 Let $G = \Spin_{2n}^\eps(q)$ with $n\ge3$. If $y\in L$ is a long root element,
 then
 \begin{enumerate}
  \item[\rm(a)] $\chi^{(0,\mu)}(y) = q^{2n-5} +\eps (q^{n-1} - q^{n-2})-1$, and
  \item[\rm(b)] $\chi^{(\neq 0,\mu)}(y) = q^{2n-5} -\eps q^{n-2}$.
 \end{enumerate}
\end{lem}

\begin{proof}
Let $\chi$ be an element of type $\eta\in\{0,\ne0\}$ from $Q^*$. By definition
$W^{(\eta,\mu)} = \mu \uparrow_{P_\chi'}^{P'}$ where $\mu$ is a linear
character of $P_\chi'$. As in the proof of Lemma~\ref{lem:WchiL B} it suffices
to compute $\chi^{(\eta,1)}(y)$.

Now $\chi^{(\eta,1)}_{L'}$ is the permutation character of $L'$ on the cosets
of $L_\chi'$. Thus $\chi^{(\eta,1)}(y)$ can be computed by counting vectors $v$
in $C_Q(y)$ with $F(v) = 0$, and with $F(v) = \alpha \neq 0$. To make the count
we observe that $C_Q(y)$ is the orthogonal direct sum of a totally singular
$2$-space with a non-degenerate space of dimension $2n-6$.

Thus the number of singular non-zero vectors in $C_Q(y)$ is equal to
$$q^2(q^{2n-7}+ \eps(q^{n-3}-q^{n-4}))- 1 = q^{2n-5} + \eps(q^{n-1}-q^{n-2})-1$$
whereas the number of vectors $v\in C_Q(y)$ with $F(v) = \alpha \neq 0$ equals
$q^2 (q^{2n-7} -\eps q^{n-4}) =  q^{2n-5} -\eps q^{n-2}$. The claim follows.
\end{proof}

\begin{prop}   \label{prop:QfixD}
 Let $G = \Spin_{2n}^\eps(q)$ with $n\ge3$. If $W$ is a $Q$-linear small
 $kG$-module then $C_W(Q) \neq \{0\}$.
\end{prop}

\begin{proof}
We argue as in the proof of Proposition~\ref{prop:QfixB}. Let $x \in Q$ and
$y \in L$ be long root elements of $G$.

As $W$ is $Q$-linear small, as a $P'$-module it decomposes as
$C_W(Q)\oplus W^0\oplus W^{\neq 0}$. Denote the sum of the multiplicities
of the characters $\chi^{(\eps,\mu)}$ in the character of $W^\eps$ by
$a_\eps$. We denote the character of the $P'$-module $C_W(Q)$ by $\chi^c$
and the character of $W$ by $\phi$. So
$\phi_{P'} = \chi^c + a_0\chi^0 + a_1\chi^{\neq 0}$ and thus
$$ \phi(x) = \chi^c(1) + a_0(\eps(q^{n-1} - q^{n-2}) - 1) -\eps a_1q^{n-2}$$
by Lemma~\ref{lem:Wchi D} and
$$\phi(y) = \chi^c(y) + (a_0+a_1)q^{2n-5} + a_0(\eps(q^{n-1} - q^{n-2}) - 1)
  - \eps a_1q^{n-2}$$
by Lemma~\ref{lem:WchiL D}. As $x,y$ are $G$-conjugate we have
$\phi(x) = \phi(y)$. Noting that $a_0+a_1 > 0$ as $W$ is faithful we see that
$$\chi^c(1) - \chi^c(y) = (a_0+a_1)q^{2n-5} \geq q^{2n-5} > 0$$
whence $C_W(Q)\ne\{0\}$.
\end{proof}

\begin{prop}   \label{prop:DnSmall}
 Let $G=\Spin_{2n}^\eps(q)$ with $\eps\in\{\pm\}$ and $n\ge3$ and let $W$ be
 an irreducible $Q$-linear small $kG$-module. Then $W$ occurs as an
 $\ell$-modular composition factor of the Harish-Chandra induction from $L$
 to $G$ of one of the modules in Table~\ref{tab:allDn}.
\end{prop}

\begin{proof}
As $W$ is $Q$-linear small Proposition~\ref{prop:QfixD} shows that
$C_W(Q) \neq 0$. As in the proof of Proposition~\ref{prop:BnSmall} this implies
that any constituent of $W|_{L'}$ has dimension not larger than
$(q^{n-1}-1)(q^{n-2}+1)$
and then we may conclude using Theorem~\ref{thm:lowDn}.
\end{proof}

\section{The main result}   \label{sec:main}
We are finally in a position to obtain the sought for gap result. For this,
we keep the notation from the previous sections. In particular $P=QL$ is
an end-node maximal parabolic subgroup of $\Spin_m^{(\pm)}(q)$ with unipotent
radical $Q$ and Levi factor $L$. Furthermore, we keep the notation $\rho_i$,
$\rho_s,\rho_t^\pm,\ldots$ for small dimensional irreducible characters as in
Section~\ref{sec:complex} with $s,t$ certain semisimple elements in $G^*$. For
an ordinary character $\chi$ we denote by $\chi^0$ its $\ell$-modular Brauer
character, that is, its restriction to $\ell$-regular classes. Throughout, for
an integer $m$, we set
$$\kappa_{\ell,m}:=\begin{cases} 1& \text{if $\ell|m$},\\
                                 0& \text{otherwise}.\end{cases}$$

\subsection{The odd-dimensional spin groups}

Let $n\ge2$, $G=\Spin_{2n+1}(q)$ and $\ell$ a prime not dividing $q$. We first
collect some results on the decomposition numbers of low dimensional ordinary
representations in non-defining characteristic.

\begin{lem}   \label{lem:rhotBn}
 Let $G=\Spin_{2n+1}(q)$, $n\ge2$, and $\eps\in\{\pm\}$. Then  $\rho_t^\eps$
 remains irreducible modulo $\ell$ when $\ell{\not|}\,(q-\eps1)$ or when
 $o(t)\ne\ell^f$ or $2\ell^f$, and otherwise
 $$(\rho_t^\eps)^0=\begin{cases}
         (\rho_1+\rho_2+\eps1_G)^0& \text{if }o(t)=\ell^f>1,\\
         (\rho_{s,q}+\eps\rho_{s,1})^0& \text{if }o(t)=2\ell^f>2,\,\ell\ne2.
 \end{cases}$$
\end{lem}

\begin{proof}
According to the description in the proof of Theorem~\ref{thm:lowBn}, the
character $\rho_t^-$ is semisimple in the Lusztig series indexed by an element
$t\in G^*$ of order dividing $q+1$. But by the observation in
\cite[Prop.~1]{HM}, the
semisimple characters in a Lusztig series $\cE(G,t)$ remain irreducible modulo
all primes $\ell$ for which the $\ell'$-part of $t$ has the same centraliser
as $t$. By our description of the parameters $t$, this is the case unless
this $\ell'$-part has order at most~2. \par
When $o(t)$ is a power of $\ell$, then by Brou\'e--Michel \cite[Thm.~9.12]{CE}
$\rho_t^-$ lies in a unipotent block, and by its explicit description in
terms of Deligne--Lusztig characters, we find that
$(\rho_t^-)^0=(\rho_1+\rho_2-1_G)^0$. Finally, if $o(t)$ is twice a power of
$\ell$, then its 2-part is conjugate to $s$ and hence $\rho_t^-$ lies in the
same $\ell$-block as the Lusztig series of $\rho_{s,1}$. Again the claim
follows from the explicit formula for the semisimple character $\rho_t^-$ in
terms of Deligne--Lusztig characters. \par
The argument for the characters  $\rho_t^+$ is entirely similar.
\end{proof}

Thus, either $\rho_t^\pm$ remains irreducible modulo~$\ell$, or its
$\ell$-modular constituents are known if we know them for the remaining
characters in Table~\ref{tab:allBn}. We therefore henceforth only consider
the latter.

\begin{lem}   \label{lem:rhosBn}
 Let $G=\Spin_{2n+1}(q)$ with $q$ odd and $n\ge3$, and $\ell$ a prime not
 dividing $q(q+1)$. Assume that the $\ell$-modular decomposition matrix of
 $\cE_\ell(G,s)$ is unitriangular. Then the entries in its first ten rows are
 approximated from above by Table~\ref{tab:E(G,s)}, where $k:=n-3$. In
 particular, both $\rho_{s,1}$ and $\rho_{s,q}$ remain irreducible modulo $\ell$.
\end{lem}

\begin{table}[ht]
{\small\[\begin{array}{c|c|cccccccccc}
  \rho& a_\rho\cr
\hline
  1\sqt1& 0& 1\\
  \St\sqt1& 1& .& 1\\
  1\sqt\rho_1& 1& .& .& 1\\
  1\sqt\rho_2& 1& .& .& .& 1\\
  1\sqt\rho_3& 1& k+1& .& .& .& 1\\
  1\sqt\rho_4& 1& k+1& .& .& .& .& 1\\
  \St\sqt\rho_1& 2& .& .& .& .& .& .& 1\\
  \St\sqt\rho_2& 2& .& .& .& .& .& .& .& 1\\
  \St\sqt\rho_3& 2& .& k+1& .& .& .& .& .& .& 1\\
  \St\sqt\rho_4& 2& .& k+1& .& .& .& .& .& .& .&1\\
\end{array}\]}
\caption{Approximate decomposition matrices for $\cE(\Spin_{2n+1}(q),s)$, $n\ge3$}
    \label {tab:E(G,s)}
\end{table}

Here, $q^{a_\rho}$ is the precise power of $q$ dividing $\rho(1)$.

\begin{proof}
By a result of Geck, see \cite[Thm.~14.4]{CE}, the Lusztig series $\cE(G,s)$ of
the semisimple involution $s\in G^*$ forms a basic set for the union of
$\ell$-blocks $\cE_\ell(G,s)$. By Lusztig's Jordan decomposition $\cE(G,s)$ is
in bijection with $\cE(C,1)$, where $C=C_{G^*}(s)\cong \Sp_2(q)\Sp_{2n-2}(q)$,
hence with $\cE(\Sp_2(q),1)\times\cE(\Sp_{2n-2}(q),1)$. Accordingly, we may
and will denote the elements of $\cE(G,s)$ by exterior tensor products of
unipotent characters, so that $\rho_{s,1}=1\sqt1$ and $\rho_{s,q}=\St\sqt1$.
\par
The character $\rho_{s,1}\in\cE(G,s)$ is semisimple, so remains irreducible
modulo all odd primes (see \cite[Prop.~1]{HM}). We next claim that the $\ell$-modular reduction of
$\rho_{s,1}$ does not occur as a composition factor of $\rho_{s,q}^\circ$.
Indeed, by the known decomposition numbers for $\Spin_5(q)\cong\Sp_4(q)$ (see
\cite{Wh90}), $\rho_{s,q}$ remains irreducible unless $\ell|(q+1)$. Now assume
the assertion has already been shown for $\Spin_{2n-1}(q)$. Thus the upper
left-hand corner of the $\ell$-modular decomposition matrix for
$\cE_\ell(\Spin_{2n-1}(q),s)$ has the form:
\[\begin{array}{l|lllll}
  \rho_{s,1}& 1\\
  \rho_{s,q}& .& 1\\
\end{array}\]
Harish-Chandra inducing the projective characters corresponding to the two
columns of this matrix yields projective characters of $G$ of the same form,
and thus, by uni-triangularity, $\rho_{s,q}^\circ$ is irreducible. The upper
bounds on the remaining entries given in Table~\ref{tab:E(G,s)} are now
obtained inductively exactly as in the proof of \cite[Thm.~6.3]{DM19} by
Harish-Chandra inducing projective characters from a Levi subgroup of an end
node parabolic subgroup.
\end{proof}

\begin{prop}   \label{prop:Bn}
 Let $G=\Spin_{2n+1}(q)$ with $q$ odd and $n\ge4$, and $\ell$ a prime not
 dividing~$q(q+1)$ such that the $\ell$-modular decomposition matrix of $G$ is
 uni-triangular Assume that $(n,q)\ne(4,3),(5,3)$. Then any $\ell$-modular
 irreducible Brauer character $\vhi$ of $G$ of degree less than
 $q^{4n-8}-q^{2n}$ is a constituent of the $\ell$-modular reduction
 of one of the complex characters listed in Table~\ref{tab:allBn}.
\end{prop}

\begin{proof}
By assumption we have that $\vhi(1)$ is smaller than the constant $b_2$ in
\cite[Table~4]{MRT}, whence by \cite[Prop.~5.3]{MRT} the module $W$ is
$Q$-linear small, unless we are in one of the exceptions listed in
\cite[Rem.~5.4]{MRT}. The only groups on that list relevant here are
$\Spin_9(3),\Spin_{11}(3)$, which were excluded. Then
Proposition~\ref{prop:BnSmall} applies to show that $\vhi$ occurs in the
$\ell$-modular reduction of a constituent $\chi$ of the Harish-Chandra
induction of some character $\psi$ of $L$ as in Table~\ref{tab:allBn}.
In particular $\chi$ is either unipotent, or in $\cE(G,s)$ or $\cE(G,t)$.
\par
If $\chi$ is unipotent, then by \cite[Cor.~6.5]{DM19} we obtain that $\vhi$
is in fact a constituent of one of the complex characters in
Table~\ref{tab:allBn}. Now assume that $\chi\in\cE(G,t)$. By the main result of
\cite{BDR17}, the union of $\ell$-blocks in $\cE_\ell(G,t)$ is Morita equivalent
to the union of unipotent $\ell$-blocks in $\cE_\ell(C,1)$, where $C$ is dual
to $C_{G^*}(t)\cong \Sp_{2n-2}(q)(q-\eps1)$. By \cite[Thm.~2.1]{GMST} the
smallest non-trivial degree of any unipotent $\ell$-modular Brauer character
of $C$ is at least $c:=\frac{1}{2}(q^{n-1}-1)(q^{n-1}-q)/(q+1)$ (observe that
the Weil modules are not unipotent as $\ell\ne2$). Hence any
$\vhi\ne(\rho_t^\eps)^\circ$ in $\cE_\ell(G,t)$ has degree at least
$$|G^*:C_{G^*}(s)|_{q'}\cdot c
  =\frac{1}{2}\,\frac{(q^{2n}-1)(q^{n-1}-1)(q^{n-1}-q)}{(q+1)^2}$$
which is larger than our bound.
\par
Finally, assume that $\chi\in\cE(G,s)$. The Harish-Chandra induction of
$\rho_{s,1}$ and $\rho_{s,q}$ from $L$ to $G$ only contains the characters
denoted $\rho_{s,1}$, $\rho_{s,q}$, $1\sqt\rho_i$ and
$\St\sqt\rho_i$, for $i=2,3,4$, from Table~\ref{tab:E(G,s)}. Since we
assume that the $\ell$-modular decomposition matrix of $G$ is uni-triangular,
lower bounds for the degrees of the corresponding Brauer characters can be
derived from Lemma~\ref{lem:rhosBn}. These show that $\vhi$ must be equal to
one of $\rho_{s,1}^\circ$, $\rho_{s,q}^\circ$.
\end{proof}

\begin{rem}   \label{rem:triangular}
The proof shows that in fact it suffices to assume that the $\ell$-modular
decomposition matrix of $G$ has a uni-triangular submatrix for the rows
corresponding to the constituents of the Harish-Chandra induction from $L$
to $G$ of the complex characters listed in Table~\ref{tab:allBn}. By
\cite[Thm.~6.3]{DM19} under mild assumptions on $\ell$ this is known for the
unipotent characters; in those cases we only need to assume it for the
characters in $\cE(G,s)$ listed in Table~\ref{tab:E(G,s)}.
\end{rem}

Let $d_\ell(q)$ denote the order of $q$ modulo~$\ell$. We then obtain Theorem~\ref{thm:main odd} in the following form:
 
\begin{thm}   \label{thm:BnSmall}
 Let $G=\Spin_{2n+1}(q)$ with $q$ odd and $n\ge4$, and $\ell\ge5$ a prime not
 dividing~$q$ such that $d_\ell(q)$ is either odd, or $d_\ell(q)>n/2$. Let
 $\vhi$ be an $\ell$-modular irreducible Brauer character of $G$ of degree less
 than $\frac{1}{2}(q^{4n-8}-q^{2n})$. Then $\vhi(1)$ is one of
 $$\begin{aligned}
  1,\quad & \frac{q^{2n}-1}{q^2-1},\quad \hlf q\frac{(q^n-1)(q^{n-1}-1)}{q+1},
  \quad \hlf q\frac{(q^n+1)(q^{n-1}+1)}{q+1},\\ 
  & \hlf q\frac{(q^n+1)(q^{n-1}-1)}{q-1}-\kappa_{\ell,q^n-1},\quad
   \hlf q\frac{(q^n-1)(q^{n-1}+1)}{q-1}-\kappa_{\ell,q^n+1},\\ 
  & q\frac{q^{2n}-1}{q^2-1},\quad 
   \frac{q^{2n}-1}{q\pm1}.
 \end{aligned}$$
\end{thm}

\begin{proof}
We claim that the assumption of Proposition~\ref{prop:Bn} is satisfied in our
situation. First, it is well-known that decomposition matrices for blocks with
cyclic defect groups are uni-triangular, so we are done in that case. Now let
$G\hookrightarrow\tG$ be a regular embedding, that is, $\tG$ is a group coming
from an algebraic group with connected centre and with the same derived
subgroup as $G$. By the result of Gruber--Hiss \cite[Thm.~8.2(c)]{GH97} the
decomposition matrix of any classical group with connected centre is
uni-triangular whenever $\ell>2$ is a linear prime. (Recall that a prime $\ell$
is \emph{linear} for $\tG$ if the order $d_\ell(q)$
of $q$ modulo~$\ell$ is odd.) Then, the proof of Proposition~\ref{prop:Bn}
shows that all $\ell$-modular Brauer characters of $\tG$ of degree less
than $q^{4n-8}-q^{2n}$ are as claimed. Now $|\tG/C_{\tG}(G)|=2$, so any
irreducible (Brauer) character of $\tG$ restricted to $G$ has at most two
irreducible constituents. Thus our claim holds for $G$ as well.
\par
It remains to discuss the groups $\Spin_9(3)$ and $\Spin_{11}(3)$ excluded in
the statement of Proposition~\ref{prop:Bn}. Their Sylow $\ell$-subgroups are
cyclic for all primes $\ell>5$, and then all small-dimensional Brauer
characters can easily be determined from the known ordinary character
degrees. For the prime $\ell=5$ we have $d_\ell(q)=d_5(3)=4$, so it is excluded
in our conclusion.
\end{proof}

\begin{rem}
For even $q$ we have $\Spin_{2n+1}(q)\cong\Sp_{2n}(q)$, and for these groups
it was shown by Guralnick--Tiep \cite[Thm.~1.1]{GT04} that the conclusion of
Theorem~\ref{thm:BnSmall} continues to hold for $n\ge5$, while for $n=4$
the lower bound has to be replaced by $q^2(q^4-1)(q^3-1)$ when $q>2$ and by
$203$ when $q=2$.
\end{rem}

\subsection{The even-dimensional spin groups}

\begin{thm}   \label{thm:DnSmall}
 Let $q$ be odd. Let either $G=\Spin_{2n}^+(q)$ with $n\ge5$, and $\ell\ge5$ a
 prime not dividing $q(q+1)$, or let $G=\Spin_{2n}^-(q)$ with $n\ge6$, and
 $\ell\ge5$ is a prime not dividing $q$. Then any
 $\ell$-modular irreducible Brauer character $\vhi$ of $G$ of degree less than
 $q^{4n-10}-q^{n+4}$ is a constituent of the $\ell$-modular reduction of one
 of the complex characters listed in Table~\ref{tab:allDn}.
\end{thm}

\begin{proof}
By comparing we see that $\vhi(1)$ is smaller than the constant $b_2$ in
\cite[Table~4]{MRT}, whence by \cite[Prop.~5.3]{MRT} the module $W$ is
$Q$-linear small, unless we are in one of the exceptions listed in
\cite[Rem.~5.4]{MRT}. Then Proposition~\ref{prop:DnSmall} applies to show that
$\vhi$ is a constituent of the $\ell$-modular reduction of the Harish-Chandra
induction of some character $\psi$ of $L$ as in Table~\ref{tab:allDn}.
But in fact, the only exception relevant here is $\Spin_{10}^+(3)$, and there
the smallest degree of a non-trivial character of
$L_\chi'=\Spin_6^+(3)=\SL_4(3)$ is~26 and $[L:L_\chi]=(3^4-1)(3^3+1)=2240$,
while the bound in the statement is $3^{20-10}-3^{5+4}=39366$, so the
conclusion holds here as well. 
\par
We consider the various possibilities. If $\psi$ is unipotent, so one of $1_L$,
$\rho_1$ or $\rho_2$, then its Harish-Chandra induction only contains
characters occurring in Table~6 or~7 of \cite{DM19}. Our claim in this case
follows from \cite[Cor.~5.8]{DM19}. \par
Next assume that $\psi$ is one of $\rho_t^-$ or $\rho_t^+$. Then its
Harish-Chandra induction lies in the Lusztig series $\cE(G,t)$. By the
main result of \cite{BDR17} the $\ell$-blocks in $\cE_\ell(G,t)$ are Morita
equivalent to the unipotent $\ell$-blocks of a group dual to
$C_{G^*}(t)\cong \CSO_{2n-2}^\eps(q)$. In particular, the decomposition
matrices are the same. For the latter we may apply \cite[Prop.~5.7]{DM19}
to see that all Brauer characters in that series apart from $\rho_t^\pm$ have
degree at least
$$\frac{(q^n-\eps1)(q^{n-1}-\eps1)}{(q+1)}
  \Big(q\frac{(q^{n-1}-\eps1)(q^{n-3}+\eps1)}{q^2-1}-1\Big)$$
which is larger than our bound. A similar argument applies to the constituents
of the Harish-Chandra induction of $\rho_{s,a}^\pm$ and $\rho_{s,b}^\pm$.
In this case, \cite{BDR17} yields a Morita equivalence between the blocks in
$\cE_\ell(G,s)$ and the unipotent blocks of the disconnected group
$\CO_{2n-2}^\eps(q)$, with connected component of index~2. Another application
of \cite[Prop.~5.7]{DM19} shows our assertion in this last case.
\end{proof}

In order to make the previous result more explicit, we determine the
$\ell$-modular reductions of some of the low-dimensional
$\Spin_{2n}^\pm(q)$-modules in Table~\ref{tab:allDn}. The first result extends
\cite[Thm.~5.5]{DM19}:

\begin{prop}   \label{prop:decmat2Dn}
 Let $G=\Spin_{2n}^-(q)$ with $q$ odd and $n\ge6$, and $\ell\ge5$ a prime
 dividing $q+1$. Then the first eight rows of the decomposition matrix of
 the unipotent $\ell$-blocks of $G$ are approximated from above by
 Table~\ref{tab:Spin-}, where $k:=n-6$.
\end{prop}

\begin{table}[ht]
{\small\[\begin{array}{l|c|ccc|c|cccc}
  \quad\rho& a_\rho\cr
\hline
       \binom{0, n}{-}& 0& 1\cr
     \binom{1, n-1}{-}& 1& k& 1\cr
   \binom{0, 1, n}{ 1}& 2& k& 1& 1\cr
\hline
     \binom{2, n-2}{-}& 2& \binom{k}{2}& k& .& 1\cr
\hline
 \binom{1, 2, n-1}{ 0}& 3& \binom{k+1}{2}& k& k& 1& 1\cr
 \binom{0, 2, n-1}{ 1}& 3& \binom{k}{2}& \binom{k+1}{2}& k& .& .& 1\cr
 \binom{0, 1, n-1}{ 2}& 3& \binom{k+1}{2}& k& k& 1& .& .& 1\cr
 \binom{0, 1, 2}{ n-1}& 3& .& .& 1& .& .& .& .& 1\cr
\hline
  & & ps& ps& ps& ps& A_1& ps& ps& .2\cr
\end{array}\]}
\caption{Approximate decomposition matrices for $\Spin_{2n}^-(q)$, $n\ge6$}
    \label {tab:Spin-}
\end{table}

\begin{proof}
This is proved along the very same lines as \cite[Thm.~5.5]{DM19}. We start
with the case $n=6$. Here, the six principal series PIMs are obtained from
the decomposition matrix of the Hecke algebra $\cH(B_5;q^2;q)$. The projective
character in the $A_1$-series comes by Harish-Chandra induction from a PIM of
a Levi subgroup of type $A_4$, while the projective character in the
``.2"-series is obtained
from a Levi subgroup of type $\tw2D_4\times A_1$. This shows the claim for
$n=6$ (with $k=0$). Then Harish-Chandra induction of these eight projective
characters yields projective characters of $G$ with the stated decompositions
for all $n\ge7$.
\par
No other Harish-Chandra series can contribute to characters of $a$-value at
most~3 by \cite[Prop.~5.3]{DM19}.
\end{proof}

\begin{rem}
For $n=5$ there is at least one unipotent PIM of $\tw2D_5(q)$ in the
Harish-Chandra series of type $A_1^2$ of $a$-value~2, and we do not see how to
rule out that there might be several of them.
\end{rem}

\begin{lem}   \label{lem:rhotDn}
 Let $G=\Spin_{2n}^\eps(q)$, $\eps\in\{\pm\}$ and $n\ge3$. Then
 $$(\rho_t^\eps)^0=\begin{cases}
       (\eps\rho_1+\rho_2+1_G)^0& \text{if }o(t)=\ell^f>1,\\
       (\rho_{s,a}^\eps+\rho_{s,b}^\eps)^0&
           \text{if }o(t)=2\ell^f>2,\,\ell\ne2,\end{cases}$$
 and $\rho_t^\eps$ remains irreducible modulo $\ell$ otherwise.
 Furthermore, $\rho_{s,a}^\pm,\rho_{s,b}^\pm$ remain irreducible modulo
 all primes $\ell\ne2$.
\end{lem}

\begin{proof}
According to the description in the proof of Theorem~\ref{thm:lowDn}, the
characters $\rho_{s,a}^\pm$ and $\rho_{s,b}^\pm$ are semisimple
in Lusztig series indexed by elements of order $2$, so we may argue as in the
proof of Lemma~\ref{lem:rhosBn} using \cite[Prop.~1]{HM}.
The proof for $\rho_t^\pm$ is completely analogous to the one of
Lemma~\ref{lem:rhotBn}.
\end{proof}

\begin{cor}   \label{cor:main Dn}
 Keep the assumptions on $n,q$ and $\ell$ from Theorem~\ref{thm:DnSmall}.
 If $\vhi$ is an $\ell$-modular irreducible Brauer character of
 $\Spin_{2n}^\eps(q)$ of degree $\vhi(1)<q^{4n-10}-q^{n+4}$ then $\vhi(1)$ is
 one of
 $$\begin{aligned}
   1,\quad
   &q\frac{(q^n-\eps1)(q^{n-2}+\eps1)}{q^2-1}-\kappa_{\ell,q^{n-1}+\eps1},\ 
   &q^2\frac{(q^{2n-2}-\eps1)}{q^2-1}-\kappa_{\ell,q^n-\eps1},\\
   &\hlf\frac{(q^n-\eps1)(q^{n-1}\pm\eps1)}{q\mp 1},\ 
   &\frac{(q^n-\eps1)(q^{n-1}\pm\eps1)}{q\mp1}.
 \end{aligned}$$
\end{cor}

\begin{proof}
This follows directly from Theorem~\ref{thm:DnSmall} with the partial
decomposition matrix for the unipotent characters in \cite[Prop.~5.7]{DM19}
and the statement of Lemma~\ref{lem:rhotDn}.
\end{proof}

This implies Theorem~\ref{thm:main even}.


\end{document}